\documentclass[12pt]{amsart}
\usepackage{amssymb}
\usepackage{amsmath}
\usepackage{mathabx}
\usepackage[bottom]{footmisc}
\numberwithin{equation}{section}
\usepackage{breakcites}
\usepackage{color}
\usepackage{graphicx}
\usepackage{epstopdf}

\newtheorem{thm}{Theorem}
\newtheorem{lemma}[thm]{Lemma}
\newtheorem{cor}[thm]{Corollary}

\theoremstyle{definition}
\newtheorem{definition}[thm]{Definition}

\newcommand{\bigperp}{\mathop{\vcenter{\hbox{\scalebox{1.5}{$\perp$}}}}\displaylimits}

\def\ff{{\mathbb F}}
\def\zz{{\mathbb Z}}
\def\mc{{\mathcal C}}
\def\ca{{\mathcal A}}
\def\zs{\mathsf{Zsig}}

\newcommand\blfootnote[1]{
    \begingroup
    \renewcommand\thefootnote{}\footnote{#1}
    \addtocounter{footnote}{-1}
    \endgroup
}

\begin{document}

\title[Invariable non-generation]{Infinitely many finite simple groups of Lie type are not generated invariably by two elements of prime order}

\begin{abstract}
We show that there are infinitely many $k$ for which $Sp_{2k}(2)$ is not generated invariably by two elements of prime order.
\end{abstract}

\author[R. M. Guralnick]{Robert M. Guralnick}
\address{Department of Mathematics, University of Southern California, 
Los Angeles, CA 90089-2532 USA}
\email{guralnic@usc.edu}

\author[J. Shareshian]{John Shareshian}
\address{Department of Mathematics, Washington University, 
St Louis, MO 63124 USA}
\email{jshareshian@wustl.edu}

\author[R. Woodroofe]{Russ Woodroofe}
\address{Univerza na Primorskem, 
Koper, Slovenia}
\email{russ.woodroofe@famnit.upr.si}

\maketitle

\blfootnote{Work of the first author was partially supported by NSF Grant DMS-1901595 and Simons Foundation
Fellowship 609771. Work
of the third author is supported in part by the Slovenian Research Agency research program P1-0285 and
research projects J1-50000, J1-70047, N1-0160, and J1-3003.}

\section{Introduction}

We say elements $g_1,g_2,\ldots,g_m$ of a group $G$ {\it generate $G$ invariably} if $\langle x_1^{-1}g_1x_1,x_2^{-1}g_2x_2,\ldots,x_m^{-1}g_mx_m \rangle=G$ for all $x_1,x_2,\ldots,x_m \in G$.  It was shown by Guralnick and Malle in \cite{GM} and independently by Kantor, Lubotzky, and Shalev in \cite{KLS} that every finite simple group is generated invariably by two elements.  This strengthens the (somewhat) classical result that every such group is generated by two elements, see \cite{Miller}, \cite{St} and \cite{AG}.  

For ordinary generation, restrictions can be placed on the orders of the generating elements: Every finite simple group is generated by two elements of prime order.  By a result of King in \cite{King}, one can take one of these elements to be an involution.  The analogous result does not hold for invariable generation.  There are infinitely many alternating groups that cannot be generated invariably by two elements of prime order (see \cite{DGHP} and \cite{SW}).  The picture for finite simple groups of Lie type has been less clear.  One can show using the Atlas \cite{atlas} and either bare hands or GAP that $\Omega_8^+(2)$ is not generated invariably by two elements of prime power order.  Similarly, $SL_4(2) \cong A_8$, $U_4(3)$, $Sp_6(2)$, $Sp_8(2)$, $\Omega_7(3)$, and $\Omega^+_8(3)$ are not generated invariably by two elements of prime order.  It turns out that the same holds for $\Omega_7(5)$, $\Omega_7(7)$, $\Omega_7(17)$, although this is harder to show and will appear in forthcoming work.  It is reasonable  to hope based on results in \cite{GM,KLS} that such examples are rare.

This hope is expressed explicitly in Section 6 of \cite{DGHP}, where it is conjectured that all but finitely many finite simple groups of Lie type are generated invariably by two elements of prime order.  Here we show that this conjecture is not true.  In particular, we have the following result, which encompasses the two symplectic examples mentioned above.

\begin{thm} \label{thm.main}
If $p$ is a prime and $a$ is a positive integer, then $Sp_{2k}(2)$ is not generated invariably by two elements of prime order whenever
\begin{enumerate}
\item $k=p^a$,
\item $k=2p^a$,
\item $k=3p^a$,
\item $k=6p^a$, or
\item $k=45$.
\end{enumerate}
\end{thm}

In forthcoming work, we will present positive results on invariable generation of finite simple groups of Lie type by two elements of restricted orders.  In some reasonable sense, most finite simple groups of Lie type are generated invariably by two elements of prime order.

\section{Proof of Theorem \ref{thm.main}}

Given a positive integer $k$, we set $G=Sp_{2k}(2)$.  Since Theorem \ref{thm.main} makes no claim when $k=1$ and holds by inspection when $k=2$ and $Sp_{2k}(2) \cong S_6$, we assume $k \geq 3$.  We write $V$ for the natural module $\ff_2^{2k}$ for $G$ and $(\cdot,\cdot)$ for the nondegenerate alternating (and therefore symmetric) form on $V$ stabilized by $G$.  Now $G$ acts on the set of nondegenerate quadratic forms $Q$ on $V$ satisfying 
\begin{equation} \label{irthform}
(v,w)=Q(v+w)+Q(v)+Q(w)
\end{equation}
 for all $v,w \in V$.  The action is given by $Q^g(v):=Q(vg^{-1})$ for all quadratic forms $Q$ under consideration, all $g \in G$, and all $v \in V$.  There are two $G$-orbits under this action.  One consists of those $Q$ such that there is some $k$-dimensional $W<V$  that is totally singular with respect to $Q$ (that is, satisfies $Q(w)=0$ for all $w \in W$). The other $G$-orbit consists of those $Q$ for which no totally singular $k$-dimensional $W<V$ exists.  Elements of the first orbit are said to be of $+$ type, and have stabilizer in $G$ isomorphic to $O^+_{2k}(2)$.  Elements of the second orbit are said to be of $-$ type, and have stabilizer in $G$ isomorphic with $O^-_{2k}(2)$.  According to \cite[Theorem 2]{Dye}, every $g \in G$ fixes at least one quadratic form satisfying (\ref{irthform}).  With this in mind, we identify elements of prime order that do not fix quadratic forms of both types.

\begin{lemma} \label{lem.bothtypes}
Assume $g \in G$ has prime order $r$.
\begin{enumerate}
\item If $r$ is odd, then $g$ fixes quadratic forms of both types if and only if $C_V(g) \neq 0$.
\item If $r=2$ and $g$ does not have Jordan type $(2,\ldots,2)$, then $g$ fixes quadratic forms of both types.
\item If $k$ is even, then $G$ has two conjugacy classes of elements of order $2$ and Jordan type $(2,2,\ldots,2)$.  Elements of one class fix a quadratic form of $+$ type but do not fix a quadratic form of $-$ type.  Elements of the other class fix quadratic forms of both types.
\item If $k$ is odd, then $G$ has one conjugacy class of elements of order $2$ and Jordan type $(2,2,\ldots,2)$.  Elements of this class fix quadratic forms of both types.
\end{enumerate}
\end{lemma}

\begin{proof}
We begin with (1).  According to \cite[Theorem 6]{Dye}, any $h \in G$ fixes an element of each orbit if and only if there is some $Q$ fixed by $h$ such that $C_V(h)$ is not totally singular with respect to $Q$. Therefore, if $C_V(g)=0$, then $g$ does not fix elements of both orbits.  On the other hand, if $C_V(g) \neq 0$, then the restriction $(\cdot,\cdot)_{C_V(g)}$ of $(\cdot,\cdot)$ to $C_V(g)$ is nondegenerate (see for example \cite[Lemma 2]{Dye}).  We write $C$ for $ C_V(g)$. Now $V=C \perp C^\perp$, with $g$ acting trivially on $C$ and preserving the restriction of $(\cdot,\cdot)$ to $C^\perp$.  It follows that the restriction $g_C$ preserves nondegenerate quadratic forms of both types that satisfy (\ref{irthform}) with respect to $(\cdot,\cdot)_C$.  Similarly, $g_{C^\perp}$ preserves some quadratic form satisfying (\ref{irthform}) with respect to $(\cdot,\cdot)_{C^\perp}$.  By \cite[Proposition 2.5.11(ii)]{KleidmanLiebeck}, $g$ preserves nondegenerate quadratic forms of both types on $V$.

Claims (2),(3), and (4) follow from \cite[Theorems 7.3 and 6.21]{LiebeckSeitzUnipotentBook}.
\end{proof}

\begin{cor} \label{cor.fpf}
If $G$ is generated invariably by two elements $g,h$ of prime order then (up to exchanging $g$ and $h$) either
\begin{enumerate}
\item $g$ and $h$ have distinct odd orders and $C_V(g)=C_V(h)=0$, or
\item $g$ has order $2$ and Jordan type $(2,2,\ldots,2)$ and fixes no quadratic form of $-$ type, while $h$ has odd order and $C_V(h)=0$.
\end{enumerate}
\end{cor}

\begin{proof}
The orders of $g$ and $h$ are distinct by Sylow's Theorem.  No element that fixes quadratic forms of both types can generate $G$ invariably with another element, since every element fixes some quadratic form.  The corollary follows from Lemma \ref{lem.bothtypes}.
\end{proof}

We record some known facts about elements $g \in G$ of odd prime order satisfying $C_V(g)=0$.  Given such $g$, assume $|g|=p$.  Then $p$ is a Zsigmondy prime for some $2^d-1$ with $2 \leq d \leq 2k$.  In fact, $d$ divides $2k$, since every nontrivial irreducible $\ff_2$-representation of $\zz_p$ has degree $d$.  (See for example \cite[Theorem 2.47(ii)]{LN}.)

\begin{lemma} \label{lem.structure}
Assume $d$ divides $2k$ and $p$ is a Zsigmondy prime for $2^d-1$.  Let $g \in G$ have order $p$ with $C_V(g)=0$.
\begin{enumerate}
\item If $d$ is even, then $V=\bigperp_{j=1}^{2k/d}\space\space V_j$, with each $V_j$ a $d$-dimensional $g$-invariant subspace on which the restriction of $(\cdot,\cdot)$ is nondegenerate and on which $\langle g \rangle$ acts irreducibly.
\item If $d$ is odd, then $V=\bigperp_{j=1}^{k/d}\space\space V_j$, with each $V_j$ a $2d$-dimensional $g$-invariant subspace on which the restriction of $(\cdot,\cdot)$ is nondegenerate.  Moreover, each $V_j =W_j \oplus X_j$ with $W_j$ and $X_j$ totally singular $d$-dimensional $g$-invariant subspaces on which $\langle g \rangle$ acts irreducibly.
\end{enumerate}
Conversely, if $d$ divides $2k$ and $p$ is a Zsigmondy prime for $2^d-1$ then $G$ has an element $g$ of order $p$ satisfying $C_V(g)=0$.
\end{lemma}

\begin{proof}
To prove (1), we proceed by induction on $2k/d$.  If $2k/d \leq 2$ then the stabilizer of an orthogonal decomposition into $2k/d$ nondegenerate subspaces of dimension $d$ contains a Sylow $p$-subgroup of $G$.  Assume now that $2k/d>2$. Pick some $d$-dimensional $W<V$ on which $\langle g \rangle$ acts nontrivially and irreducibly.  By irreducibility, $W$ is either nondegenerate or totally singular with respect to $(\cdot,\cdot)$.  In the first case we apply the inductive hypothesis to $W^\perp$.  In the second case, by \cite[Lemma 4.1.12]{KleidmanLiebeck}, there exists $d$-dimensional $Y<V$ such that $\langle g \rangle$ acts nontrivially and irreducibly on $Y$ and $W \oplus Y$ is nondegenerate with respect to $(\cdot,\cdot)$.  We apply the inductive hypothesis to the proper subspaces $W \oplus Y$ and $(W \oplus Y)^\perp$.   

To prove the converse, it suffices to prove it in the case $d=2k$.  To do this, we observe that the order of the stabilizer of a $1$-space from $V$ is not divisible by $p$, see for example \cite[Proposition 4.1.19]{KleidmanLiebeck}.

Claim (2) and its converse follow from \cite[Lemma 4.1.12]{KleidmanLiebeck} and induction on dimension.
\end{proof}

The following well-known technical lemma will be useful.  Say $d=\ell m$ and $q$ is a prime power.  Then $\ff_{q^\ell}^m$ and $\ff_q^d$ are isomorphic as $\ff_q$-vector spaces, and the natural action of $GL_m(q^\ell)$ on $\ff_{q^\ell}^m$ is $\ff_q$-linear.  This gives an embedding $\iota:GL_m(q^\ell) \hookrightarrow GL_d(q)$.  Moreover, if $H \leq GL_m(q^\ell)$ preserves a nondegenerate bilinear, sesquilinear, or quadratic form on $\ff_{q^\ell}^m$ then $\iota(H)$ preserves a form of the same type on $\ff_q^d$.

\begin{lemma} \label{lem.extensions}
Assume $d=\ell m$ and $g \in GL_m(q^\ell)$ and let ${\mathbb G}$ be the Galois group of the extension $\ff_q \hookrightarrow \ff_{q^\ell}$.
\begin{enumerate}
\item If $g$ (acting on $\ff_{q^\ell}^m$) has characteristic polynomial $\Phi_g(X) \in \ff_{q^\ell}[X]$ then $\iota(g)$ (acting on $\ff_q^d$) has characteristic polynomial $$\prod_{\alpha \in {\mathbb G}} \alpha(\Phi_g(X)) \in \ff_q[X].$$
\item If $g$ (acting on $\ff_{q^\ell}^m$) is unipotent of Jordan type $\lambda$ then the Jordan type of $\iota(g)$ (acting on $\ff_q^d$) is obtained by concatenating $\ell$ copies of $\lambda$ and permuting entries.
\end{enumerate}
\end{lemma}

\begin{proof}
Write $\overline{\ff_q}$ for the algebraic closure of $\ff_q$ and let $I:GL_m(q^\ell) \rightarrow GL_m(\overline{\ff_q})$ be the identity embedding.  Both claims follow from the fact that the representation of $GL_{m}(q^\ell)$ on $\ff_q^d \otimes \overline{\ff_q}$ determined by $\iota$ is the direct sum of the representations $I \circ \alpha$ over all $\alpha \in {\mathbb G}$ (see for example \cite[Lemma 2.10.2(ii)]{KleidmanLiebeck}).
\end{proof}

Assume $\ell$ is a prime divisor of $2k$.  Unless $\ell=2$ and $k$ is odd, $GL_{2k/\ell}(2^\ell)$ has subgroups $H \cong Sp_{2k/\ell}(2^\ell)$.  We fix one such $H$.  As noted above, $\iota(H) \leq G$, and the normalizer $M_3(\ell)$ of $\iota(H)$ in $G$ is isomorphic with $Sp_{2k/\ell}(2^\ell):\zz_\ell$ (see \cite[Proposition 4.3.10]{KleidmanLiebeck}). So, the groups $M_3(\ell)$ are the maximal subgroups lying in Aschbacher class ${\mathcal C}_3$.

We will see that all of the elements described in Corollary \ref{cor.fpf} are contained in some conjugate of some $M_3(\ell)$, and some of them are contained in a conjugate of every $M_3(\ell)$.

\begin{lemma} \label{lem.c3}
Let $p$ be a Zsigmondy prime for $2^d-1$ with $d$ dividing $2k$.
\begin{enumerate}
\item If odd prime $\ell$ divides $d$ then $M_3(\ell)$ contains a Sylow $p$-subgroup of $G$.
\item If $k$ is even and $d$ is divisible by $4$ then $M_3(2)$ contains a Sylow $p$-subgroup of $G$.
\end{enumerate}
\end{lemma}

\begin{proof}
Assume first that $\ell$ is odd.  Since $\gcd(2^d-1,2^a-1)=2^{\gcd(d,a)}-1$, the $p$-part of $|G|$ is the $p$-part of $\prod_{d|2j \leq 2k}(2^{2j}-1)=\prod_{d|2j \leq 2k}((2^\ell)^{2j/\ell}-1)$, which divides $\prod_{i=1}^{k/\ell}(2^{2i\ell}-1)$, and thererfore divides $|M_3(\ell)|$.  Similarly, if $k$ is even and $4$ divides $d$, then the $p$-part of $|G|$ is the $p$-part of $\prod_{d|2j \leq 2k}(4^j-1)$.  Since $4$ divides $d$, if $d$ divides $2j$ then $j$ is even.  So, $\prod_{d|2j \leq 2k}(4^j-1)$ divides $\prod_{i=1}^{k/2}(4^{2i}-1)$, which divides $|M_3(2)|$.
\end{proof}

We observe that if $d>2$, then either some odd prime divides $d$ or $4$ divides $d$.  Therefore, if $p$ is a Zsigmondy prime for $2^d-1$ and $g \in G$ has order $p$, then some $M_3(\ell)$ contains a conjugate of $g$.  The case $d=2$ is a special case of the next lemma.

\begin{lemma} \label{lem.everym3}
Assume that $d$ divides $2k$ and that $p$ is a Zsigmondy prime for $2^d-1$.  If either
\begin{itemize}
\item[(a)] $d$ is even and $p=d+1$ or
\item[(b)] $d$ is odd and $p=2d+1$,
\end{itemize}
then there is a unique conjugacy class $\mc(p)$ in $G$ such that every $g \in \mc(p)$ has order $p$ and satisfies $C_V(g)=0$.  Moreover, $\mc(p)$ intersects every $M_3(\ell)$ nontrivially.
\end{lemma}

\begin{proof}
By Lemma \ref{lem.structure}, $G$ contains some element $g_p$ of order $p$ satisfying $C_V(g_p)=0$.

If $p=d+1$ (so $d$ is even), then there is only one nontrivial irreducible $\ff_2$-representation of $\zz_p$, call it $\phi$.  If $x \in \zz_p \setminus \{1\}$ then the eigenvalues of $\phi(x)$ are the $d$ primitive $p^{th}$ roots of $1$ in the algebraic closure $\overline{\ff_2}$. It follows that, in the language of Lemma \ref{lem.structure}, $g_p$ acts on each $V_j$ with $d$ distinct eigenvalues.  This argument applies to every element $g \in G$ having order $p$ and satisfying $C_V(g)=0$.  Therefore all such $g$ have the same multiset of eigenvalues (each primitive $p^{th}$ appearing $2k/d$ times) and are conjugate in $GL_{2k}(2)$.  It follows that all such $g$ are conjugate in $G$ (see \cite[Theorem 6.1.1]{DLO}).

If $p=2d+1$ with $d$ odd, then $\zz_p$ has two nontrivial irreducible representations $\phi_1,\phi_2$, both of degree $d$.  The $\phi_i$ are dual to each other.  (The Galois group $\zz_d$ of the extension $\ff_2 \hookrightarrow \ff_{2^d}$ acts on the character group of $\zz_p$.  The fact that $d$ is odd precludes a nontrivial character and its inverse from being in the same orbit.) We can now apply Lemma \ref{lem.structure} as we did in the case $p=d+1$ to get the same result.

Finally, we show that each $M_3(\ell)$ intersects $\mc(p)$, the conjugacy class of $g_p$.  If $\ell$ is odd and divides $d$ or if $\ell=2$ and $4$ divides $d$ then the claim follows from Lemma \ref{lem.c3}.  Otherwise, $d$ divides the even number $2k/\ell$ and there is some $x \in Sp_{2k/\ell}(2) \leq Sp_{2k/\ell}(2^\ell)$ with each eigenvalue a primitive $p^{th}$ root of $1$ and all such roots having the same multiplicity as eigenvalues.  Our claim follows now from Lemma \ref{lem.extensions}(1).
\end{proof}

\begin{cor} \label{cor.everyzsig}
If $d$ divides $2k$, $p$ is a Zsigmondy prime for $2^d-1$, and $g \in G$ has order $p$ with $C_V(g)=0$, then there is some $\ell$ such that $M_3(\ell)$ contains a conjugate of $g$.
\end{cor}

\begin{proof}
As noted above, if $d>2$ then some $M_3(\ell)$ contains a Sylow $p$-subgroup of $G$ by Lemma \ref{lem.c3}.  If $d=2$, then $p=3$, and every $M_3(\ell)$ contains a conjugate of $g$ by Lemma \ref{lem.everym3}.
\end{proof}

\begin{lemma} \label{lem.invm3}
Assume $k$ is even.  If $u \in G$ is unipotent with Jordan type $(2,2,\ldots,2)$ and fixes no quadratic form of $-$ type then $M_3(\ell)$ contains a conjugate of $u$ whenever $\ell$ is odd.  Moreover, if $4$ divides $k$ then $M_3(2)$ contains a conjugate of $u$.
\end{lemma}

\begin{proof}
Using the notation from \cite[Chapter 6]{LiebeckSeitzUnipotentBook}, we have
$$
V\downarrow_u=W(2)^{k/2}.
$$
If $\ell$ is odd or if $k$ is divisible by $4$ and $\ell=2$ then $k/\ell$ is even and $Sp_{2k/\ell}(2^\ell)$ contains an element $u_\ell$ satisfying
$$
(\ff_{2^\ell}^{2k/\ell})\downarrow_{u_\ell}=W(2)^{k/2\ell}.
$$
By Lemma \ref{lem.extensions}(2), $\iota(u_\ell) \in G$ has Jordan type $(2,2,\ldots,2)$.  Moreover, as noted in the proof of Lemma \ref{lem.extensions}, over $\overline{\ff_2}$ we can write $\iota(u_\ell)$ as a direct sum of matrices of the form $\alpha(u_\ell)$, with $\alpha$ running through the Galois group ${\mathbb G}$ of the extension $\ff_2 \hookrightarrow \ff_{2^\ell}$.  It remains to show that each $\alpha(u_\ell)$ is conjugate with $u_\ell$ in $Sp_{2k/\ell}(2^\ell)$.   By \cite[Lemma 6.2 and Theorem 7.3]{LiebeckSeitzUnipotentBook}, if $\alpha(u_\ell)$ is not conjugate with $u_\ell$ then
$$
(\ff_{2^\ell}^{2k/\ell})\downarrow_{\alpha(u_\ell)}=W(2)^{k/2\ell-1} \oplus V(2)^2.
$$
By construction (see \cite[Section 6.1]{LiebeckSeitzUnipotentBook}), $(v,u_\ell v)=0$ for all $v \in \ff_{2^\ell}^{2k/\ell}$.  The same must hold for $\alpha(u_\ell)$, but does not hold for unipotent $x$ such that $(\ff_{2^\ell}^{2k/\ell}) \downarrow_x$ has $V(2)$ as a component. 
\end{proof}

\begin{cor} \label{cor.oddonly}
If $x,y \in G$ have prime order and generate $G$ invariably, then both $x$ and $y$ have odd order.
\end{cor}

\begin{proof}
We must show that if $u \in G$ is an involution of Jordan type $(2,2,\ldots,2)$ that fixes no quadratic form of $-$ type, then there is no $y \in G$ of prime order that together with $u$ generates $G$ invariably.  By Corollary \ref{cor.fpf} and Sylow's Theorem, any such $y$ has odd order $p$, where $p$ is a Zsigmondy prime for some $2^d-1$ with $d$ dividing $2k$.  Moreover, $C_V(y)=0$.  Assume we are given such a $y$.  If $d$ is not a power of $2$, or if $4$ divides $k$, then some $M_3(\ell)$ contains a conjugate of $u$ and a conjugate of $y$ by Corollary \ref{cor.everyzsig} and Lemma \ref{lem.invm3}.  It remains to handle the cases where $k/2$ is odd and $d=2$ or $d=4$.  In these cases the only Zsigmondy prime for $2^d-1$ is $d+1$.  We see from the proof of Lemma \ref{lem.everym3} that $y$ fixes a nondegenerate $4$-dimensional subspace of $V$.  As $u$ also fixes such a subspace (consider one component $W(2)$ in $V\downarrow_u$), $u$ and $y$ cannot generate $G$ invariably.
\end{proof}

With Lemma \ref{lem.everym3} (along with the fact that there is no Zsigmondy prime for $2^6-1$) in mind, we make the following definitions.

\begin{definition}
Given a positive integer $d$, we write $\zs_2(d)$ for the set of all Zsigmondy primes for $2^d-1$, and define
$$
\ca_{\mathsf{even}}:=\left\{d \in 2\zz_{>0}:\zs_2(d)=\{d+1\}\right\},
$$
$$
\ca_{\mathsf{odd}}:=\left\{d \in (2\zz+1)_{>1}:\zs_2(d)=\{2d+1\}\right\},
$$
and
$$\ca:=\ca_{\mathsf{even}} \cup \ca_{\mathsf{odd}} \cup \{6\}.
$$
\end{definition}

So, $\ca$ consists of all $d$ such that $2^d-1$ admits at most one Zsigmondy prime, and any such Zsigmondy prime is the smallest odd number $\ell>1$ satisfying $\ell \equiv 1 \bmod d$.  Calculation yields the following result.

\begin{lemma} \label{lem.ca}
We have $\{2,3,4,6,10,12,18\} \subseteq \ca$.
\end{lemma}

We are ready to prove Theorem \ref{thm.main}.  By Corollaries \ref{cor.fpf} and \ref{cor.oddonly}, if $g,h \in G$ have respective prime orders $p,r$ and generate $G$ invariably, then there exist nontrivial divisors $c,d$ of $2k$ such that $p,r$ are respective Zsigmondy primes for $2^c-1$ and $2^d-1$.  Moreover, $C_V(g)=C_V(h)=0$.  If $\gcd(c,d)>2$ then there is some $\ell$ such that $M_3(\ell)$ contains both a conjugate of $g$ and a conjugate of $h$ by Lemma \ref{lem.c3}, which precludes invariable generation.  So, $\gcd(c,d) \leq 2$.

Assume now that $k=bp^a$ with $b \in \{1,2,3,6\}$, $p$ a prime, and $a \geq 1$.  Say $p$ is odd.  If $c,d$ are as above and $\gcd(c,d) \leq 2$, then one of $c,d$ is not divisible by $p$ and therefore divides $12$ and must lie in $\ca$.  Lemma \ref{lem.everym3} and Corollary \ref{cor.everyzsig} preclude invariable generation.  If $p=2$ then we may adjust $a$ if necessary so that $b \in \{1,3\}$.  One of $c,d$ is not divisible by $4$ and therefore lies in $\{2,3,6\}$.  Again we can apply Lemma \ref{lem.everym3} and Corollary \ref{cor.everyzsig}.

We are left with the case $k=45$.  We check that if $c,d$ are distinct nontrivial divisors of $90$, neither in $\ca$, and $\gcd(c,d) \leq 2$, then $\{c,d\}=\{5,9\}$.  By Lemma \ref{lem.structure}, if $p,r$ are respective Zsigmondy primes for $2^5-1$ and $2^9-1$ and $g,h \in G$ have respective orders $p,r$ with $C_V(g)=C_V(h)=0$, then there are totally singular subspaces $X,Y \leq V$ both of dimension $45$ from $V$, such that $g$ and $h$ stabilize $X$ and $Y$, respectively.  (When $d=5$, consider the subspace $\perp_{i=1}^9 W_i$ from Lemma \ref{lem.structure}(2), and reason similarly when $c=9$.)  Therefore $g$ and $h$  cannot generate $G$ invariably.

\section*{Acknowledgment}
Claude was used to proofread text and check mathematical arguments.


\begin{thebibliography}{}

\bibitem{AG} M. Aschbacher and R. Guralnick, Some Applications of the First Cohomology Group, {\it J. Algebra} {\bf 90} (1984), no. 2, 446-460.


\bibitem{atlas} J. H. Conway, R. T. Curtis, S. P. Norton, R. A. Parker and R. A. Wilson, {\it Atlas of Finite Groups}, Clarendon Press, Oxford, 1985.

\bibitem{DLO} G. De Franceschi, M. W. Liebeck, and E. A. O'Brien, {\it Conjugacy in Finite Classical Groups}, Springer, Cham, Switzerland, 2025.

\bibitem{DGHP} S. Dolfi, R. M. Guralnick, M. Herzog, and C. Praeger, A new solvability criterion for finite groups,
{\it J. Lond. Math. Soc.} (2) {\bf 85} (2012), no. 2, 269-281.

\bibitem{Dye} R. H. Dye, Interrelations of symplectic and orthogonal groups in characteristic two, {\it J. Algebra} {\bf 59} (1979), 202-221.

\bibitem{GM} R. M.  Guralnick and G. Malle, Simple groups admit Beauville structures, {\it J. London Math. Soc.} {\bf 85} (2012), 694-721. 

\bibitem{KLS} W. M. Kantor, A. Lubotzky, and A. Shalev, Invariable generation and the Chebotarev invariant of a finite group, {\it J. Algebra} {\bf 348}  (2011) 302-314.

\bibitem{King} C. S. H. King, Generation of finite simple groups by an involution and an element of prime order, {\it J. Algebra} {\bf 478} (2017), 153-173.

\bibitem{KleidmanLiebeck} P. Kleidman and M. Liebeck, {\it The Subgroup Structure of the Finite Classical Groups}, London Mathematical Society Lecture Note Series {\bf 129}, Cambridge University Press, 1990.

\bibitem{LN} R.Lidl and H. Niederreiter, {\it Finite Fields} (foreword by P. M. Cohn), Cambridge University Press, Cambridge, 1997.

\bibitem{LiebeckSeitzUnipotentBook} M. W. Liebeck and G. M. Seitz, {\it Unipotent and Nilpotent Classes in Simple Algebraic Groups and Lie Algebras}, Mathematical Surveys and Monographs {\bf 180}, American Mathematical Society, Providence, 2012. 

\bibitem{Miller} G. A. Miller, On the groups generated by two operators, {\it Bull. Amer. Math. Soc.} {\bf 7} (1901), no. 10, 424-426.

\bibitem{SW} J. Shareshian and R. Woodroofe, Divisibility of binomial coefficients and generation of alternating groups, {\it Pacific J. Math.} {\bf 292} (2018), no. 1, 223-238.

\bibitem{St} R. Steinberg, Generators for simple groups, {\it Canadian J. Math.} {\bf 14} (1962), 277-283. 

\end{thebibliography}
\end{document}